\documentclass{article}

\usepackage[english]{babel}
\usepackage{graphicx}
\usepackage{textcomp}
\usepackage{amssymb}
\usepackage{amsmath}
\usepackage{setspace}
\usepackage{geometry}
\usepackage{xcolor}
\usepackage[numbers,sort&compress]{natbib}

\newtheorem{corollary}{Corollary}
\newtheorem{remark}{Remark}
\newtheorem{definition}{Definition}
\newtheorem{theorem}{Theorem}
\newtheorem{lemma}{Lemma}

\newtheorem{question}{Question}
\newenvironment{proof}[1][Proof:]{\begin{trivlist}
\item[\hskip \labelsep {\bfseries #1}]}{\end{trivlist}}

\begin{document}

\title{Elements (functions) that are universal with respect to a minimal system}

\author{Zhirayr Avetisyan\footnote{Department of Mathematics: Analysis, Logic and Discrete Mathematics,
  Ghent University, Belgium, zhirayr.avetisyan@ugent.be} \and Martin Grigoryan\footnote{Chair of Higher Mathematics, Physics Faculty, Yerevan State University. gmarting@ysu.am} \and Michael Ruzhansky\footnote{Department of Mathematics: Analysis, Logic and Discrete Mathematics,
  Ghent University, Belgium, and
  School of Mathematical Sciences,
    Queen Mary University of London,
     United Kingdom. m.ruzhansky@ugent.be
}}

\maketitle

\begin{abstract}
We call an element $U$ conditionally universal for a sequential convergence space $\mathbf{\Omega}$ with respect to a minimal system $\{\varphi_n\}_{n=1}^\infty$ in a continuously and densely embedded Banach space $\mathcal{X}\hookrightarrow\mathbf{\Omega}$ if the partial sums of its phase-modified Fourier series is dense in $\mathbf{\Omega}$. We will call the element $U$ almost universal if the change of phases (signs) needs to be performed only on a thin subset of Fourier coefficients. In this paper we prove the existence of an almost universal element under certain assumptions on the system $\{\varphi_n\}_{n=1}^\infty$. We will call a function $U$ asymptotically conditionally universal in a space $L^1(\mathcal{M})$ if the partial sums of its phase-modified Fourier series is dense in $L^1(F_m)$ for an ever-growing sequence of subsets $F_m\subset\mathcal{M}$ with asymptotically null complement. Here we prove the existence of such functions $U$ under certain assumptions on the system $\{\varphi_n\}_{n=1}^\infty$. Moreover, we show that every integrable function can be slightly modified to yield such a function $U$.

In particular, we establish the existence of almost universal functions for $L^p([0,1])$, $p\in(0,1)$, and asymptotically conditionally universal functions for $L^1([0,1])$, with respect to the trigonometric system.
\end{abstract}

\section{Introduction}

Unsurprisingly, the term ``universal'' in mathematics is used in many different contexts and means different things depending on the context. In functional analysis and approximation theory, perhaps the most common use of the term is associated with the concept of a universal object, such as a function or a sequence, which is capable of representing a large class of objects in a certain way. There are many different ways or senses which a function, a sequence or a series may be universal in, and the study of existence and structure of such universal objects is an old and classical subject. A small selection of works in this context from different times can be found in \cite{Bir29,EpMu14,Fek14,GaGr21,GaGr18,Gro87,Gri99,GrSa16,Gri17,GeNa99,Gri20,Gri20a,Gri22,GrSa17,Iva90,Kro91} and \cite{Luh86,Men64,Men42,Mar35,Mac52,MeNe01,Nes96,Tal60,BGNP08}.

Perhaps the first important instance of a function with universal properties is due to Birkhoff \cite{Bir29}, who proved the existence of an entire function $f$ such that for an arbitrary entire function $g$ there exists a subsequence of natural numbers $\{n_k\}_{k=1}^\infty\subset\mathbb{N}$ such that the functional sequence $\{f(\cdot+n_k)\}_{k=1}^\infty$ converges to $g$ uniformly on compact subsets of $\mathbb{C}$. Hence, the sequence of additive translations $\{f(\cdot+n)\}_{n=1}^\infty$ of the function $f$ is dense in the space of entire functions endowed with the topology of compact convergence.

In \cite{Mar35} Marcinkiewicz proved that for any non-zero null sequence $\{h_n\}_{n=1}^\infty\subset\mathbb{R}\setminus\{0\}$, $h_n\xrightarrow[n\to\infty]{}0$, there exists a real continuous function $F_h\in C([0,1],\mathbb{R})$ having the following property: for every measurable function $g$ there exists a subsequence $\{n_k\}_{k=1}^\infty\subset\mathbb{N}$ such that
$$
\frac{F_h(x+h_{n_k})-F_h(x)}{h_{n_k}}\xrightarrow[k\to\infty]{}g(x),\quad\mbox{a.e.}\,x\in[0,1].
$$
This continuous function $F_h$ is called the universal primitive function w.r.t. the sequence $\{h_n\}_{n=1}^\infty$.

In \cite{Gro87} Gro{\ss}e-Erdmann proved the existence of a smooth function $f\in C^\infty(\mathbb{R},\mathbb{R})$, $f(0)=0$, with a locally uniformly universal in $C(\mathbb{R},\mathbb{R})$ Taylor series at $x_0=0$. That is, for every $g\in C(\mathbb{R},\mathbb{R})$ with $g(0)=0$ and $r>0$ there exists a subsequence of natural numbers $\{n_k\}_{k=1}^\infty\subset\mathbb{N}$ such that the corresponding partial sums of the Taylor series of $f$ converges to $g$ uniformly for $x\in[-r,r]$,
$$
\sup_{|x|<r}\left|\sum_{m=1}^{n_k}\frac{f^{(m)}(0)}{m!}x^m-g(x)\right|\xrightarrow[k\to\infty]{}0.
$$

In a similar vein, Luh proved in \cite{Luh86} that for every $r>0$ there exists a power series $\sum\limits_{n=0}^\infty a_nz^n$ with radius of convergence $r$ such that for every compact $K\Subset\{z\in\mathbb{C}\,\vline\,|z|>r\}$, with connected complement $K^\complement$, and every function $g\in C(K)\cap\mathrm{Hol}(\mathring{K})$ continuous in $K$ and holomorphic in the interior $\mathring{K}$, there exists a subsequence of natural numbers $\{n_k\}_{k=1}^\infty\subset\mathbb{N}$ such that the corresponding partial sums of the power series converges to $g$ uniformly in $K$,
$$
\sup_{z\in K}\left|\sum_{m=1}^{n_k}c_mz^m-g(z)\right|\xrightarrow[k\to\infty]{}0.
$$

The notion of a universal series appropriate to the present paper is due to Menshov \cite{Men64} and Talalian \cite{Tal60}. Since its inception the subject of universal series has grown considerably, with some of the influential contributions being \cite{EpMu14}, \cite{Gri99}, \cite{GeNa99}, \cite{Iva90}. In this paper we will be interested in functions (or more general elements) which have a universal Fourier series with respect to a given system, with universality being interpreted in several ways. We will introduce the setting and formulate the problem in its utmost generality in Definition \ref{OrthSerUniversalDef} and Definition \ref{FuncUniversalDef} below. In simple terms, we are given a Banach space $\mathcal{X}$ continuously and densely embedded in another space $\mathbf{\Omega}$, together with a minimal system $\{\varphi_n\}_{n=1}^\infty\subset\mathcal{X}$. The problem is to find an element $U\in\mathcal{X}$ such that its Fourier series with respect to $\{\varphi_n\}_{n=1}^\infty$ exhibits universal properties in the space $\mathbf{\Omega}$. In Theorem \ref{TheoremAbstract} we will show the existence of such an element $U$ under the assumption that $\mathbf{\Omega}$ is a homogeneous metric space and that the system $\{\varphi_n\}_{n=1}^\infty$ possesses what we call the universal approximation property in Definition \ref{UniApproxPropDef}. This element $U$ has the property of almost universality, namely, that, after a change of phase (sign) in a subset of terms of lower density zero, the partial sums of its Fourier series become dense in $\mathbf{\Omega}$.  We will discuss the ramifications of this theorem for the most common situation where $\mathcal{X}=L^1(\mathcal{M})$ and $\mathbf{\Omega}=L^p(\mathcal{M})$, $p\in(0,1)$, for a suitable measure space $\mathcal{M}$.

It is clear that the uniqueness of Fourier coefficients precludes any universality properties in the space $\mathbf{\Omega}=L^1(\mathcal{M})$ itself. The slightly weaker condition of asymptotic conditional universality in $L^1(\mathcal{M})$ is proven to be satisfied by a function $U$ in Theorem \ref{AsymTheorem1} under the assumption that the system $\{\varphi_n\}_{n=1}^\infty$ possesses what we call the asymptotic approximation property in Definition \ref{AsymApproxPropDef}. Moreover, in Theorem \ref{AsymTheorem2} it is shown that such asymptotically conditionally universal functions $U$ are in good supply. Namely, every integrable function can be modified on a subset of arbitrarily small positive measure to yield an asymptotically conditionally universal function. Here asymptotic conditional universality is understood as denseness, after a change of phases (signs), of partial sums of the Fourier series in each $L^1(F_m)$, where $\{F_m\}_{m=1}^\infty$ is a growing sequence of subsets $F_m\subset\mathcal{M}$ which in the limit covers almost all of $\mathcal{M}$, i.e., $|F_m^\complement|\to0$.

Finally, these general results are applied to the particular case of the trigonometric system on the interval $[-\pi,\pi]$, to establish the existence of almost universal functions in $L^p([-\pi,\pi])$, $p\in(0,1)$ (and hence also in $L^0([-\pi,\pi])$ and $L([-\pi,\pi])$; see the discussion including Remark \ref{L1MRemark} below), as well as the existence and abundance of asymptotically conditionally universal functions in $L^1([-\pi,\pi])$. These results appear to be new.

\section{The setting}

Let us introduce the key concepts of universality, in the context relevant to the present paper, in a somewhat abstract setting. We will need the following ingredients:

\begin{itemize}

\item $\mathbf{\Omega}$ a Hausdorff sequential convergence space

\item $\mathcal{X}$ a Banach space, with a continuous and dense embedding in $\mathcal{X}\hookrightarrow\mathbf{\Omega}$

\item $\{\varphi_n\}_{n=1}^\infty$ a minimal system in $\mathcal{X}$

\end{itemize}

We begin with the simple notion of a universal sequence.

\begin{definition} A sequence $\{\xi_n\}_{n=1}^\infty\subset\mathcal{X}$ is called \textbf{universal} in $\mathbf{\Omega}$ if subsequences of $\{\xi_n\}_{n=1}^\infty$ converge to every point of $\mathbf{\Omega}$,
$$
\left(\forall f\in\mathbf{\Omega}\right)\left(\exists\{N_k\}_{k=1}^\infty\subset\mathbb{N}\right)\quad\xi_{N_k}\xrightarrow[k\to\infty]{\mathbf{\Omega}}f.
$$
\end{definition}

In this paper we will be concerned with universality in the sense of convergence of certain sequences, therefore only sequential convergence will be relevant for our purposes. If the space $\mathbf{\Omega}$ is any convergence space then only its induced sequential convergence structure will figure in the present context. It is therefore no loss of generality to assume the space $\mathbf{\Omega}$ to be sequential. We have begun with the extreme generality of a Hausdorff sequential convergence space $\mathbf{\Omega}$. But per Proposition 1.7.15 in \cite{BeBu02}, the mere assumption of the very mild Urysohn property immediately makes $\mathbf{\Omega}$ topological. Nevertheless, one of the most natural sequential convergences, namely, the almost everywhere convergence on a measure space, is not topological. In order to not leave this convergence out of our scope from the very onset we will allow general Hausdorff sequential convergence spaces.

On the other hand, universality of a sequence in a first-countable topological space has the following continuity property with respect to continuous dense embeddings.

\begin{lemma}\label{UsualUnivContLemma} Let $\mathcal{X}\hookrightarrow\mathbf{\Omega}\hookrightarrow\mathbf{\Theta}$ be a sequence of continuous dense embeddings, where $\mathcal{X}$ is a Banach space and $\mathbf{\Omega}$, $\mathbf{\Theta}$ are Hausdorff first-countable topological spaces. Suppose that the sequence $\{\xi_n\}_{n=1}^\infty\subset\mathcal{X}$ is universal in $\mathbf{\Omega}$. Then it is universal in $\mathbf{\Theta}$.
\end{lemma}
\begin{proof} Let $f\in\mathbf{\Theta}$ and a basis $\{O_\ell\}_{\ell=1}^\infty$ of open neighbourhoods of $f$ in $\mathbf{\Theta}$ be given. Assume without loss of generality that $O_{\ell+1}\subset O_\ell$ for $\forall\ell\in\mathbb{N}$. By denseness of $\mathbf{\Omega}\hookrightarrow\mathbf{\Theta}$, let $\{f_m\}_{m=1}^\infty\subset\mathbf{\Omega}$ be a sequence such that $f_m\xrightarrow[m\to\infty]{\mathbf{\Theta}}f$. For every $\ell\in\mathbb{N}$, take $m_\ell\in\mathbb{N}$ such that $f_{m_\ell}\in O_\ell$, so that $O_\ell$ is an open neighbourhood of $f_{m_\ell}$. By universality of $\{\xi_n\}_{n=1}^\infty$ in $\mathbf{\Omega}$ find a subsequence $\{N_k^\ell\}_{k=1}^\infty\subset\mathbb{N}$ such that
$\xi_{N_k^\ell}\xrightarrow[k\to\infty]{\mathbf{\Omega}}f_{m_\ell}$, which implies by continuity of $\mathbf{\Omega}\hookrightarrow\mathbf{\Theta}$ that $\xi_{N_k^\ell}\xrightarrow[k\to\infty]{\mathbf{\Theta}}f_{m_\ell}$. Then $\exists k_\ell\in\mathbb{N}$ such that $\forall k>k_\ell$ we have $\xi_{N_k^\ell}\in O_\ell$. Set $M_0\doteq1$ and define $\{M_\ell\}_{\ell=1}^\infty\subset\mathbb{N}$ inductively by $M_\ell\doteq N_k^\ell$ such that $k>k_\ell$ and $M_\ell>M_{\ell-1}$ for $\forall\ell\in\mathbb{N}$. Thus,
$$
\left(\forall\ell\in\mathbb{N}\right)\left(\forall\ell'>\ell\right)\quad\xi_{M_{\ell'}}\in O_{\ell'}\subset O_\ell,
$$
which shows that $\xi_{M_\ell}\xrightarrow[\ell\to\infty]{\mathbf{\Theta}}f$. This completes the proof. $\Box$
\end{proof}

In this paper we will concentrate on very particular sequences, namely, partial sums of series with respect to (w.r.t. hereinafter) the chosen minimal system $\{\varphi_n\}_{n=1}^\infty$. The minimality of $\{\varphi_n\}_{n=1}^\infty$ is equivalent to the existence of a system $\{c_n\}_{n=1}^\infty$ in the dual space $\mathcal{X}^*$ such that $(\{\varphi_n\}_{n=1}^\infty,\{c_n\}_{n=1}^\infty)$ is biorthonormal,
$$
c_n(\varphi_m)=\delta_{n,m},\quad\forall n,m\in\mathbb{N}.
$$
For a finite or infinite sequence of numbers $\{\alpha_n\}_{n=1}^N\subset\mathbb{C}$, $N\in\mathbb{N}_\infty=\mathbb{N}\cup\{\infty\}$, we will consider polynomials or series
$$
\sum_{n=1}^N\alpha_n\varphi_n
$$
w.r.t. the system $\{\varphi_n\}_{n=1}^\infty$. For an element $f\in\mathcal{X}$ we will call the elements of the sequence $\{c_n(f)\}_{n=1}^\infty$ the Fourier coefficients of $f$, and the corresponding series
$$
\sum_{n=1}^\infty c_n(f)\varphi_n
$$
the Fourier series of $f$ w.r.t. the system $\{\varphi_n\}_{n=1}^\infty$, regardless of its possible convergence either in $\mathcal{X}$ or $\mathbf{\Omega}$.

The concept of universality of a series can be summarized as follows. A series
$$
\sum_{n=1}^\infty\alpha_n\varphi_n
$$
will be called universal in $\mathbf{\Omega}$ if its partial sums,
$$
S_N\doteq\sum_{n=1}^N\alpha_n\varphi_n,\quad\forall N\in\mathbb{N},
$$
or certain modifications thereof, exhibit denseness properties in $\mathbf{\Omega}$. In other words, arbitrary elements of $\mathbf{\Omega}$ can be approximated with the help of a single fixed series $\sum_{n=1}^\infty\alpha_n\varphi_n$ in one way or another. Let us make these concepts more precise through definitions.

Let us choose a bounded subset $\mathbb{G}$, $1\in\mathbb{G}\subset\mathbb{C}$. The usual choice is $\mathbb{G}=\{\pm1\}$, in which case one speaks of signs, but we will allow for more flexible choices of $\mathbb{G}$ and will speak of $\mathbb{G}$-values.

\begin{definition}\label{OrthSerUniversalDef} Let $\mathbf{\Omega}$, $\mathcal{X}$ and $\{\varphi_n\}_{n=1}^\infty$ be as before. Then a series $\sum_{n=1}^\infty\alpha_n\varphi_n$ is said to be:

\begin{itemize}

\item[1.] \textbf{universal} in $\mathbf{\Omega}$ in the \textbf{usual} sense if its partial sums are universal in $\mathbf{\Omega}$, i.e.,
$$
\left(\forall f\in\mathbf{\Omega}\right)\left(\exists\{N_k\}_{k=1}^\infty\subset\mathbb{N}\right)\quad\sum_{n=1}^{N_k}\alpha_n\varphi_n\xrightarrow[k\to\infty]{}f
$$

\item[2.] \textbf{universal} in $\mathbf{\Omega}$ in the sense of $\mathbb{G}$-\textbf{values} (\textbf{signs}) if
$$
\left(\forall f\in\mathbf{\Omega}\right)\left(\exists\{\epsilon_n\}_{n=1}^\infty\right)\quad\left(\forall n\in\mathbb{N}\right)\,\epsilon_n\in\mathbb{G}\quad\wedge\quad\sum_{n=1}^\infty\epsilon_n\alpha_n\varphi_n=f
$$

\item[3.] \textbf{universal} in $\mathbf{\Omega}$ in the sense of \textbf{rearrangements} if
$$
\left(\forall f\in\mathbf{\Omega}\right)\left(\exists\sigma\in\mathrm{Aut}(\mathbb{N})\right)\quad\sum_{n=1}^\infty\alpha_{\sigma(n)}\varphi_{\sigma(n)}=f
$$

\end{itemize}
\end{definition}

\begin{remark} It follows from Lemma \ref{UsualUnivContLemma} that, if $\mathcal{X}\hookrightarrow\mathbf{\Omega}\hookrightarrow\mathbf{\Theta}$ are as therein, and the series $\sum_{n=1}^\infty\alpha_n\varphi_n$ is universal in $\mathbf{\Omega}$ in the usual sense, then it is such also in $\mathbf{\Theta}$.
\end{remark}

Definitions \ref{OrthSerUniversalDef} pertain to the universality of a series regardless of whether it is the Fourier series of an element $U\in\mathcal{X}$ or not. In case of the Fourier series of an $U\in\mathcal{X}$, any universality properties are attributed to the element $U$.

\begin{definition}\label{FuncUniversalDef} Let $\mathbf{\Omega}$, $\mathcal{X}$ and $\{\varphi_n\}_{n=1}^\infty$ be as before. Then we will call an element $U\in\mathcal{X}$:

\begin{itemize}

\item[1.] \textbf{universal} for $\mathbf{\Omega}$ w.r.t. $\{\varphi_n\}_{n=1}^\infty$ in the \textbf{usual} sense if the Fourier series $\sum_{n=1}^\infty c_n(U)\varphi_n$ is universal in $\mathbf{\Omega}$ in the usual sense

\item[2.] \textbf{universal} for $\mathbf{\Omega}$ w.r.t. $\{\varphi_n\}_{n=1}^\infty$ in the sense of $\mathbb{G}$-\textbf{values} (\textbf{signs}) if the Fourier series $\sum_{n=1}^\infty c_n(U)\varphi_n$ is universal in $\mathbf{\Omega}$ in the sense of $\mathbb{G}$-values (signs)

\item[3.] \textbf{universal} for $\mathbf{\Omega}$ w.r.t. $\{\varphi_n\}_{n=1}^\infty$ in the sense of \textbf{rearrangements} if the Fourier series $\sum_{n=1}^\infty c_n(U)\varphi_n$ is universal in $\mathbf{\Omega}$ in the sense of rearrangements

\item[4.] \textbf{conditionally universal} for $\mathbf{\Omega}$ w.r.t. $\{\varphi_n\}_{n=1}^\infty$ if
$$
\left(\exists\{\delta_n\}_{n=1}^\infty\right)\quad\left(\forall n\in\mathbb{N}\right)\delta_n\in\mathbb{G}
$$
and the series $\sum_{n=1}^\infty\delta_nc_n(U)\varphi_n$ is universal in $\mathbf{\Omega}$ in the usual sense

\item[5.] \textbf{almost universal} for $\mathbf{\Omega}$ w.r.t. $\{\varphi_n\}_{n=1}^\infty$ if
$$
\left(\exists\{\delta_n\}_{n=1}^\infty\right)\quad\left(\forall n\in\mathbb{N}\right)\delta_n\in\mathbb{G}\quad\wedge\quad\limsup_{n\to\infty}\frac{\#\left\{m\in\mathbb{N}\,\vline\quad m\le n\,\wedge\,\delta_m=1\right\}}{n}=1
$$
and the series $\sum_{n=1}^\infty\delta_nc_n(U)\varphi_n$ is universal in $\mathbf{\Omega}$ in the usual sense

\end{itemize}
\end{definition}

\begin{remark} It follows from Lemma \ref{UsualUnivContLemma} that, if $\mathcal{X}\hookrightarrow\mathbf{\Omega}\hookrightarrow\mathbf{\Theta}$ are as therein, and $U\in\mathcal{X}$ is universal for $\mathbf{\Omega}$ either in the usual sense or conditionally, or almost universal, then it is such also for $\mathbf{\Theta}$.
\end{remark}

\paragraph*{Discussion:} Let us examine here the relevance of some aspects of our setting above. First, $\mathcal{X}\hookrightarrow\mathbf{\Omega}$ must be dense, so that series have a chance of being universal in $\mathbf{\Omega}$. However, the topology of $\mathbf{\Omega}$ must be strictly weaker than that of $\mathcal{X}$, because convergence of a certain subsequence of partial sums of a series in the topology of $\mathcal{X}$ uniquely fixes the coefficients of the series as the Fourier coefficients of the limit function; this excludes any kind of universality of the series.

\begin{lemma}\label{FourierCoeffUniqueLemma} If any subsequence of the partial sums of the series $\sum_{n=1}^\infty\alpha_n\varphi_n$ converges to $f\in \mathcal{X}$ then $\sum_{n=1}^\infty\alpha_n\varphi_n$ is the Fourier series of the function $f$, i.e., $\alpha_n=c_n(f)$ for $n\in\mathbb{N}$.
\end{lemma}
\begin{proof} Let $\{N_k\}_{k=1}^\infty\subset\mathbb{N}$ be a subsequence such that
$$
\left\|\sum_{n=1}^{N_k}\alpha_n\varphi_n-f\right\|_\mathcal{X}\xrightarrow[k\to\infty]{}0.
$$
Take any $m\in\mathbb{N}$ and let $N_k\ge m$. Then
$$
|\alpha_m-c_m(f)|=\left|c_m\left(\sum_{n=1}^{N_k}\alpha_n\varphi_n-f\right)\right|\le\|c_m\|_{\mathcal{X}^*}\cdot\left\|\sum_{n=1}^{N_k}\alpha_n\varphi_n-f\right\|_\mathcal{X}\xrightarrow[k\to\infty]{}0,
$$
which shows that $\alpha_m=c_m(f)$. $\Box$
\end{proof}

Second, the system has to be minimal in order that we can speak of unambiguously defined Fourier coefficients $c_n(f)$ of an element $f\in\mathcal{X}$. Finally, the necessity to consider various alternative kinds of universality arises mainly from the fact that for some of the most important classical systems it is known that functions universal in the usual sense do not exist (e.g., Remark \ref{UnivNotExistRemark} below). However, functions universal in various senses defined above have been shown to exist in different classical systems \cite{GaGr21,GaGr18,GrSa16,Gri17,Gri20,Gri20a,Gri22,GrSa17}. Nevertheless, the following questions remain open.

\begin{question} Does there exist a non-trivial triple $\mathbf{\Omega}$, $\mathcal{X}$ and $\{\varphi_n\}_{n=1}^\infty$ which admits a universal element $U\in\mathcal{X}$ for $\mathbf{\Omega}$ w.r.t. $\{\varphi_n\}_{n=1}^\infty$ in the usual sense?
\end{question}
\begin{question} Does there exist a non-trivial triple $\mathbf{\Omega}$, $\mathcal{X}$ and $\{\varphi_n\}_{n=1}^\infty$ which admits a universal element $U\in\mathcal{X}$ for $\mathbf{\Omega}$ w.r.t. $\{\varphi_n\}_{n=1}^\infty$ in the sense of rearrangements?
\end{question}
The non-triviality means that the topology of $\mathbf{\Omega}$ is not too weak. Good choices would be $\mathcal{X}=L^1(\mathcal{M})$ and $\mathbf{\Omega}=L^p(\mathcal{M})$, $p\in[0,1)$, for a finite separable diffuse measure space $\mathcal{M}$, such as $\mathcal{M}=[0,1]$ (see the discussion in Section \ref{Section:Lebesgue}). To the best of our knowledge, even for $\mathcal{M}=[0,1]$ the above questions are open.

\section{Universality in metric spaces}

In this section we will produce elements $U\in\mathcal{X}$ which are almost universal for a given minimal system, under certain technical assumptions on the latter.

\begin{definition}\label{UniApproxPropDef} Let $(\mathbf{\Omega},d)$ be a metric space and $\mathcal{X}$ a Banach space, with a continuous dense embedding $\mathcal{X}\hookrightarrow\mathbf{\Omega}$. Let further $(\{\varphi_n\}_{n=1}^\infty,\{c_n\}_{n=1}^\infty)$ be a biorthonormal system in $(\mathcal{X},\mathcal{X}^*)$. We will say that the system $\{\varphi_n\}_{n=1}^\infty$ possesses the \textbf{universal approximation property} in $\mathbf{\Omega}$ if for a dense subset $\mathcal{D}\subset\mathcal{X}$ we have
$$
\left(\forall f\in\mathcal{D}\right)\left(\forall\epsilon,\delta>0\right)\left(\forall n_0\in\mathbb{N}\right)\left(\exists N\in\mathbb{N}\right)\left(\exists \{\alpha_n\}_{n=n_0}^N\subset\mathbb{C}\right)\left(\exists \{\delta_n\}_{n=n_0}^N\subset\mathbb{G}\right)
$$
such that:
\begin{itemize}

\item[1.] $\|H\|_\mathcal{X}<\epsilon$, $H\doteq\sum_{n=n_0}^N\alpha_n\varphi_n$

\item[2.] $d(f,Q)<\delta$, $Q\doteq\sum_{n=n_0}^N\delta_n\alpha_n\varphi_n$

\end{itemize}
\end{definition}

The first important result of the present work is the following theorem.

\begin{theorem}\label{TheoremAbstract} Let $(\mathbf{\Omega},d)$ be a separable Abelian metric group with an invariant metric, $(\mathcal{X},\|\cdot\|)$ a Banach space with a continuous dense additive embedding $\mathcal{X}\hookrightarrow\mathbf{\Omega}$, and $(\{\varphi_n\}_{n=1}^\infty,\{c_n\}_{n=1}^\infty)$ a bi-orthonormal system in $(\mathcal{X},\mathcal{X}^*)$. If the system $\{\varphi_n\}_{n=1}^\infty$ possesses the universal approximation property in $\mathbf{\Omega}$ then $\exists U\in\mathcal{X}$ such that $U$ is almost universal for $\mathbf{\Omega}$ w.r.t. $\{\varphi_n\}_{n=1}^\infty$.
\end{theorem}
\begin{proof} By separability of $\Omega$ choose a dense sequence $\{f_m\}_{m=1}^\infty\subset\Omega$. Using mathematical induction, we construct the sequences of polynomials $\{H_k\}_{k=1}^\infty$ and $\{Q_k\}_{k=1}^\infty$ in the system $\{\varphi_n\}_{n=1}^\infty$ as follows. Set $M_1\doteq1$. On each step $k\in\mathbb{N}$, using the denseness of the sequence $\{f_m\}_{m=1}^\infty$, find $m_k\in\mathbb{N}$ such that
\begin{equation}
d\left(f_k-\sum_{\ell=1}^{k-1}Q_\ell\,,f_{m_k}\right)<\frac1{2k}.\label{Est12k1}
\end{equation}
Then apply the universal approximation property in Definition \ref{UniApproxPropDef} with substitution
$$
f\leftarrow f_{m_k},\quad\epsilon\leftarrow\frac1{2^k},\quad\delta\leftarrow\frac1{2k},\quad n_0\leftarrow M_k
$$
to obtain the number $M_k^*\in\mathbb{N}+M_k$, the tuples
$$
\left\{\alpha^{(k)}_n\right\}_{n=M_k}^{M_k^*}\subset\mathbb{C},\quad\left\{\delta^{(k)}_n\right\}_{n=M_k}^{M_k^*}\subset\mathbb{G},
$$
and polynomials
$$
H_k\doteq\sum_{n=M_k}^{M_k^*}\alpha^{(k)}_n\varphi_n,\quad Q_k\doteq\sum_{n=M_k}^{M_k^*}\delta^{(k)}_n\alpha^{(k)}_n\varphi_n,
$$
such that
\begin{equation}
\|H_k\|_\mathcal{X}\le\frac1{2^k},\quad d(f_{m_k},Q_k)<\frac1{2k}.\label{HkQkEst}
\end{equation}
Finally, set $M_{k+1}\doteq2^kM_k^*$. Note from (\ref{Est12k1}) and (\ref{HkQkEst}) that the sequence $\{Q_k\}_{k=1}^\infty$ constructed this way has the property that
\begin{equation}
d\left(f_k,\sum_{\ell=1}^kQ_\ell\right)=d\left(f_k-\sum_{\ell=1}^{k-1}Q_\ell\,,Q_k\right)\le d\left(f_k-\sum_{\ell=1}^{k-1}Q_\ell\,,f_{m_k}\right)+d(f_{m_k},Q_k)<\frac1k,\label{ApproxEq}
\end{equation}
where we used the invariance of the metric $d$ in the first step.

Define now the sequences $\{\alpha_n\}_{n=1}^\infty\subset\mathbb{C}$ and $\{\delta_n\}_{n=1}^\infty\subset\mathbb{G}$ by
$$
\alpha_n\doteq\begin{cases}
0 & \mbox{for}\quad n\in\bigcup\limits_{k=1}^\infty[M_k^*,M_{k+1})\\
\alpha^{(k)}_n & \mbox{for}\quad n\in[M_k,M_k^*],\quad k\in\mathbb{N}
\end{cases},\quad\delta_n\doteq\begin{cases}
1 & \mbox{for}\quad n\in\bigcup\limits_{k=1}^\infty[M_k^*,M_{k+1})\\
\delta^{(k)}_n & \mbox{for}\quad n\in[M_k,M_k^*],\quad k\in\mathbb{N}
\end{cases}.
$$
Observe that
$$
\frac{\#\{n\in\mathbb{N}\,\vline\quad n\le M_{k+1},\quad\delta_n=1\}}{M_{k+1}}\ge\frac{M_{k+1}-M_k^*}{M_{k+1}}=1-\frac1{2^k},\quad\forall k\in\mathbb{N},
$$
so that
\begin{equation}
\limsup_{n\to\infty}\frac{\#\{m\in\mathbb{N}\,\vline\quad m\le n\,\wedge\,\delta_m=1\}}{n}=1.\label{UpperDensity}
\end{equation}
Set
$$
U\doteq\sum_{k=1}^\infty H_k=\sum_{n=1}^\infty\alpha_n\varphi_n.
$$
From (\ref{HkQkEst}) it follows that
$$
\sum_{k=1}^\infty\|H_k\|_\mathcal{X}<1,
$$
therefore $U\in\mathcal{X}$. Moreover,
$$
\lim_{k\to\infty}\left\|\sum_{n=1}^{M_k^*}\alpha_n\varphi_n-U\right\|_\mathcal{X}=\lim_{k\to\infty}\left\|\sum_{\ell=k+1}^\infty H_\ell\right\|_\mathcal{X}=0,
$$
so that by Lemma \ref{FourierCoeffUniqueLemma} the series $\sum_{n=1}^\infty\alpha_n\varphi_n$ is the Fourier series of the element $U$,
$$
\alpha_n=c_n(U),\quad\forall n\in\mathbb{N}.
$$
Take any $f\in\mathbf{\Omega}$ and by denseness find a subsequence $\{f_{m_q}\}_{q=1}^\infty\subset\{f_m\}_{m=1}^\infty$ such that $f_{m_q}\xrightarrow[q\to\infty]{\mathbf{\Omega}}f$. Then using (\ref{ApproxEq}) we find that
$$
d\left(f,\sum_{n=1}^{M_{m_q}^*}\delta_n\alpha_n\varphi_n\right)\le d\left(f,f_{m_q}\right)+d\left(f_{m_q},\sum_{n=1}^{M_{m_q}^*}\delta_n\alpha_n\varphi_n\right)
$$
$$
=d\left(f,f_{m_q}\right)+d\left(f_{m_q},\sum_{k=1}^{m_q}Q_k\right)<d\left(f,f_{m_q}\right)+\frac1{m_q}\xrightarrow[q\to\infty]{}0,
$$
which shows that the element $U$ is almost universal for $\mathbf{\Omega}$. $\Box$
\end{proof}

\begin{remark} The assumption of the additive structure in $\mathbf{\Omega}$ and invariance of the metric $d$ were made for technical purposes. The invariance can be easily relaxed to, say, $d(\cdot+h,\cdot\cdot+h)<A\cdot d(\cdot,\cdot\cdot)$ for all $h\in\mathcal{X}$. It is possible that these assumptions can be dropped altogether.
\end{remark}

\section{On universality in Lebesgue spaces}\label{Section:Lebesgue}

Let us now restrict to universality in spaces $\mathbf{\Omega}=M(\mathcal{M})$ (all measurable functions) or $\mathbf{\Omega}=L^p(\mathcal{M})$, $p\in[0,+\infty)$, with $\mathcal{X}=L^q(\mathcal{M})$, $q\ge1$, $p<q$, for a finite separable diffuse measure space $\mathcal{M}=(\mathcal{M},\Sigma,|\cdot|)$. First we note that since $L^r(\mathcal{M})\hookrightarrow L^q(\mathcal{M})$ is continuous and dense for $1\le q\le r$, a minimal system in $L^r(\mathcal{M})$ contains a minimal system in $L^q(\mathcal{M})$. Therefore, without loss of generality we can assume $q=\max\{1,p\}$. However, if $p\ge 1$ then $\mathbf{\Omega}=\mathcal{X}=L^p(\mathcal{M})$, and by Lemma \ref{FourierCoeffUniqueLemma} there cannot exist any universal properties. Thus, we conclude that $p\in[0,1)$ and $q=1$ is the correct setting.

Speaking of $M(\mathcal{M})$ and $L^0(\mathcal{M})$, there are two commonly used natural convergences in these spaces: convergence almost everywhere and convergence in measure. Convergence almost everywhere is not topological and is a typical representative of a sequential convergence space \cite{BeBu02}. Convergence in measure is the topologization or topological modification of convergence almost everywhere, and is in fact metrizable with a metric that is invariant with respect to the additive structure. Convergence in measure is therefore the more natural structure in $L^0(\mathcal{M})$ when viewed as a limit case of $L^p(\mathcal{M})$, $p>0$.

In fact, as far as universal functions in the usual sense, conditionally universal or almost universal functions are concerned, universality in $M(\mathcal{M})$ and $L^0(\mathcal{M})$ is independent of whether the almost everywhere convergence or the convergence in measure is chosen. Indeed, convergence almost everywhere of a sequence implies convergence in measure, while convergence in measure of a sequence implies almost everywhere convergence of a subsequence. And since the universality of a functional series in the usual sense is based on subsequences, it does not depend on the choice between these two convergences. And once there is no difference in this choice, it is more convenient to choose the metrizable convergence in measure.

On the other hand, if $L^0(\mathcal{M})$ and $M(\mathcal{M})$ are equipped with the invariant metric topology of convergence in measure, then $\mathcal{X}\hookrightarrow L^p(\mathcal{M})\hookrightarrow L^0(\mathcal{M})$ and $\mathcal{X}\hookrightarrow L^p(\mathcal{M})\hookrightarrow M(\mathcal{M})$ satisfy the assumptions of Lemma \ref{UsualUnivContLemma} for $p>0$. Therefore a function that is universal in $L^p(\mathcal{M})$ in any of these three senses is automatically universal in the same sense in $L^0(\mathcal{M})$ and $M(\mathcal{M})$. We conclude with the following

\begin{remark}\label{L1MRemark} For a function $U\in L^1(\mathcal{M})$, the universality in $L^0(\mathcal{M})$ and $M(\mathcal{M})$ in the usual sense, conditional, and almost universality are independent of whether convergence almost everywhere is assumed or convergence in measure. If $U$ is universal in any of these senses in some $L^p(\mathcal{M})$, $p\in(0,1)$, then it is universal in the same sense in any $L^q(\mathcal{M})$, $q\in[0,p)$ and in $M(\mathcal{M})$.
\end{remark}

However, universality in the sense of $\mathbb{G}$-values (signs) or rearrangements do not work with subsequences, and such convenient reductions are not available. These kinds of universality may potentially depend on the choice of convergence in $L^0(\mathcal{M})$ and $M(\mathcal{M})$.

Coming back to universality in the usual sense, conditional and almost universality, we see that the optimal strategy is to try to prove universality for a pair $\mathcal{X}=L^1(\mathcal{M})\hookrightarrow L^p(\mathcal{M})=\mathbf{\Omega}$ with $p\in(0,1)$ as large as possible. By Theorem \ref{TheoremAbstract}, it is sufficient to show that the system $\{\varphi_n\}_{n=1}^\infty$ has the universal approximation property in $\mathbf{\Omega}=L^p(\mathcal{M})$. Lemma \ref{TrigUniApproxPropLemma} below establishes this property for the trigonometric system, but for other systems the validity of this property is unknown.

\section{Asymptotic universality in $L^1(\mathcal{M})$}

While the strength of the Banach space structure in Lemma \ref{FourierCoeffUniqueLemma} precludes universality in $L^1(\mathcal{M})$ in the sense defined before, a weaker, asymptotic form of universality can be established under favourable conditions. Here $\mathcal{M}=(\mathcal{M},\Sigma,|\cdot|)$ is a finite separable diffuse measure space, as before.

\begin{definition}\label{FuncAsymUniversalDef} Let $\{\varphi_n\}_{n=1}^\infty\subset L^1(\mathcal{M})$ be a minimal system. We will call an element $U\in L^1(\mathcal{M})$:

\begin{itemize}

\item[1.] \textbf{asymptotically universal} for $L^1(\mathcal{M})$ w.r.t. $\{\varphi_n\}_{n=1}^\infty$ in the \textbf{usual} sense if there exists a sequence of subsets $\{F_m\}_{m=1}^\infty\subset2^\mathcal{M}$, with
\begin{equation}
F_1\subset F_2\subset\ldots\subset\mathcal{M},\quad\lim\limits_{m\to\infty}\left|F_m^\complement\right|=0,
\end{equation}
such that
$$
\left(\forall f\in L^1(\mathcal{M})\right)\left(\exists\{N_q\}_{q=1}^\infty\subset\mathbb{N}\right)\left(\forall m\in\mathbb{N}\right)\quad\lim\limits_{q\to\infty}\int\limits_{F_m}\left|\sum_{n=1}^{N_q}c_n(U)\varphi_n(x)-f(x)\right|dx=0
$$

\quad

\item[2.] \textbf{asymptotically conditionally universal} for $L^1(\mathcal{M})$ w.r.t. $\{\varphi_n\}_{n=1}^\infty$ if there exist a sequence of $\mathbb{G}$-values (signs) $\{\delta_n\}_{n=1}^\infty$, $(\forall n\in\mathbb{N})\,\delta_n\in\mathbb{G}$, and a sequence of subsets $\{F_m\}_{m=1}^\infty\subset2^\mathcal{M}$, with
\begin{equation}
F_1\subset F_2\subset\ldots\subset\mathcal{M},\quad\lim\limits_{m\to\infty}\left|F_m^\complement\right|=0,\label{F_mDef}
\end{equation}
such that
$$
\left(\forall f\in L^1(\mathcal{M})\right)\left(\exists\{N_q\}_{q=1}^\infty\subset\mathbb{N}\right)\left(\forall m\in\mathbb{N}\right)\quad\lim\limits_{q\to\infty}\int\limits_{F_m}\left|\sum_{n=1}^{N_q}\delta_nc_n(U)\varphi_n(x)-f(x)\right|dx=0
$$

\end{itemize}
\end{definition}

In this paper we will construct asymptotically conditionally universal functions under certain assumptions on the minimal system $\{\varphi_n\}_{n=1}^\infty$ as described in the definition below.

\begin{definition}\label{AsymApproxPropDef} Let $\{\varphi_n\}_{n=1}^\infty\subset L^1(\mathcal{M})$ be a minimal system. We will say that the system $\{\varphi_n\}_{n=1}^\infty$ possesses the \textbf{asymptotic approximation property} in $L^1(\mathcal{M})$ if for a dense subset $\mathcal{D}\subset L^1(\mathcal{M})$ and a positive number $C>0$ we have
\begin{eqnarray}
\left(\forall f\in\mathcal{D}\right)\left(\forall\epsilon,\delta,\sigma>0\right)\left(\forall n_0\in\mathbb{N}\right)\nonumber\\
\left(\exists N\in\mathbb{N}+n_0\right)\left(\exists \{\alpha_n\}_{n=n_0}^N\subset\mathbb{C}\right)\left(\exists \{\delta_n\}_{n=n_0}^N\subset\mathbb{G}\right)\left(\exists E\subset\mathcal{M}\right)
\end{eqnarray}
such that:
\begin{itemize}

\item[1.] $\left|E^\complement\right|<\sigma$

\item[2.] $\|H\|_1<\epsilon$, $H\doteq\sum_{n=n_0}^N\alpha_n\varphi_n$

\item[3.] $\int\limits_E|f(x)-Q(x)|dx<\delta$, $Q\doteq\sum_{n=n_0}^N\delta_n\alpha_n\varphi_n$

\item[4.] $\int\limits_{E^\complement}|Q(x)|dx\le C\cdot\|f\|_1$

\end{itemize}
\end{definition}

An equivalent formulation of the above definition is the following statement, which is technically more convenient for applications.

\begin{lemma}\label{AsymApproxLemma} The system $\{\varphi_n\}_{n=1}^\infty$ possess the asymptotic approximation property in $L^1(\mathcal{M})$ if and only if for a dense subset $\mathcal{D}\subset L^1(\mathcal{M})$ and a positive number $C>0$ we have
\begin{eqnarray}
\left(\forall f\in\mathcal{D}\right)\left(\forall\epsilon,\delta,\sigma>0\right)\left(\forall n_0\in\mathbb{N}\right)\nonumber\\
\left(\exists N\in\mathbb{N}+n_0\right)\left(\exists \{\alpha_n\}_{n=n_0}^N\subset\mathbb{C}\right)\left(\exists \{\delta_n\}_{n=n_0}^N\subset\mathbb{G}\right)\left(\exists E\subset\mathcal{M}\right)\left(\exists\hat g\in L^1(\mathcal{M})\right)
\end{eqnarray}
such that:
\begin{itemize}

\item[1.] $\left|E^\complement\right|<\sigma$

\item[2.] $\|H\|_1<\epsilon$, $H\doteq\sum_{n=n_0}^N\alpha_n\varphi_n$

\item[3.] $(\forall x\in E)$ $f(x)=\hat g(x)$

\item[4.] $\|\hat g\|_1\le C\|f\|_1$

\item[5.] $\int\limits_E|f(x)-Q(x)|dx=\|\hat g-Q\|_1<\delta$, $Q\doteq\sum_{n=n_0}^N\delta_n\alpha_n\varphi_n$

\end{itemize}
\end{lemma}
\begin{proof} It suffices to take $\hat g(x)=f(x)$ for $x\in E$ and $\hat g(x)=Q(x)$ for $x\notin E$. $\Box$
\end{proof}

\begin{remark}\label{DL1Remark} It is not difficult to show that the dense subset $\mathcal{D}$ in Definition \ref{AsymApproxPropDef} and Lemma \ref{AsymApproxLemma} can be replaced by the entire space $L^1(\mathcal{M})$.
\end{remark}

We are ready to present the second important result of this paper. Recall that to a minimal system $\{\varphi_n\}_{n=1}^\infty\subset L^1(\mathcal{M})$ we associate the bi-orthogonal system $(\{\varphi_n\}_{n=1}^\infty,\{c_n\}_{n=1}^\infty)$ for the dual pair $(L^1(\mathcal{M}),L^\infty(\mathcal{M}))$.

\begin{theorem}\label{AsymTheorem1} If the minimal system $\{\varphi_n\}_{n=1}^\infty$ possesses the asymptotic approximation property in $L^1(\mathcal{M})$ then there exists an integrable function $U\in L^1(\mathcal{M})$ which is asymptotically conditionally universal for $L^1(\mathcal{M})$.
\end{theorem}
\begin{proof} By separability of $L^1(\mathcal{M})$ choose a dense sequence $\{f_k\}_{k=1}^\infty\subset L^1(\mathcal{M})$. Using mathematical induction, we construct sequences of coefficients $\{\alpha_n\}_{n=1}^\infty\subset\mathbb{C}$, of $\mathbb{G}$-values (signs) $\{\beta_n\}_{n=1}^\infty\subset\mathbb{G}$, of polynomials $\{H_k^{(1)}\}_{k=1}^\infty$, $\{H_k^{(2)}\}_{k=1}^\infty$, $\{Q_k^{(1)}\}_{k=1}^\infty$ and $\{Q_k^{(2)}\}_{k=1}^\infty$ in the system $\{\varphi_n\}_{n=1}^\infty$, of subsets $\{E_k^{(1)}\}_{k=1}^\infty$ and $\{E_k^{(2)}\}_{k=1}^\infty$, and of integrable functions $\{\hat g_k\}_{k=1}^\infty$ as follows. Set $M_1\doteq1$. On each step $k\in\mathbb{N}$, first use Lemma \ref{AsymApproxLemma} (taking into account Remark \ref{DL1Remark}) with substitutions
$$
f\leftarrow f_k,\quad\epsilon\leftarrow\frac1{2^{k+1}},\quad\delta\leftarrow\frac1{2^{k+1}},\quad\sigma\leftarrow\frac{1}{2^{k+1}},\quad n_0\leftarrow M_{2k-1}
$$
to obtain the number $M_{2k}\in\mathbb{N}+M_{2k-1}$, the tuples
$$
\left\{\alpha_n\right\}_{n=M_{2k-1}}^{M_{2k}-1}\subset\mathbb{C},\quad\left\{\beta_n\right\}_{n=M_{2k-1}}^{M_{2k}-1}\subset\mathbb{G},
$$
the polynomials
\begin{equation}
H_k^{(1)}\doteq\sum_{n=M_{2k-1}}^{M_{2k}-1}\alpha_n\varphi_n,\quad Q_k^{(1)}\doteq\sum_{n=M_{2k-1}}^{M_{2k}-1}\beta_n\alpha_n\varphi_n,\label{FourierPoly1}
\end{equation}
the subset $E_k^{(1)}\subset\mathcal{M}$ and the function $\hat g_k\in L^1(\mathcal{M})$ such that
$$
|E_k^{(1)}|>|\mathcal{M}|-\frac{1}{2^{k+1}},\quad\left\|H_k^{(1)}\right\|_1<\frac1{2^{k+1}},
$$
\begin{equation}
\|\hat g_k\|_1\le C\|f_k\|_1,\quad\left\|\hat g_k-Q_k^{(1)}\right\|_1<\frac1{2^{k+1}},\quad\hat g_k|_{E_k^{(1)}}=f_k|_{E_k^{(1)}}.\label{IndPart1}
\end{equation}
Next, apply the same Lemma \ref{AsymApproxLemma} (again in view of Remark \ref{DL1Remark}) with substitutions
$$
f\leftarrow f_k-Q_k^{(1)}-\sum_{j=1}^{k-1}\left(H_j^{(1)}+Q_j^{(2)}\right),\quad\epsilon\leftarrow\frac1{2^{k+1}},\quad\delta\leftarrow\frac{1}{k},\quad\sigma\leftarrow\frac{1}{2^{k+1}},\quad n_0\leftarrow M_{2k}
$$
to obtain the number $M_{2k+1}\in\mathbb{N}+M_{2k}$, the tuples
$$
\left\{\alpha_n\right\}_{n=M_{2k}}^{M_{2k+1}-1}\subset\mathbb{C},\quad\left\{\beta_n\right\}_{n=M_{2k}}^{M_{2k+1}-1}\subset\mathbb{G},
$$
the polynomials
\begin{equation}
H_k^{(2)}\doteq\sum_{n=M_{2k}}^{M_{2k+1}-1}\alpha_n\varphi_n,\quad Q_k^{(2)}\doteq\sum_{n=M_{2k}}^{M_{2k+1}-1}\beta_n\alpha_n\varphi_n,\label{FourierPoly2}
\end{equation}
and the subset $E_k^{(2)}\subset\mathcal{M}$ such that
\begin{equation}
|E_k^{(2)}|>|\mathcal{M}|-\frac{1}{2^{k+1}},\quad\left\|H_k^{(2)}\right\|_1<\frac1{2^{k+1}},\quad\int\limits_{E_k^{(2)}}\left|f_k(x)-\sum_{j=1}^k\left(H_j^{(1)}(x)+Q_j^{(2)}(x)\right)\right|dx<\frac{1}{k}.\label{IndPart2}
\end{equation}

Now set
\begin{equation}
F_m\doteq\bigcap_{k=m}^\infty E_k^{(2)},\quad\forall m\in\mathbb{N}.\label{FmDef}
\end{equation}
It is clear that
$$
|F_m^\complement|\le\sum_{k=m}^\infty\left|(E_k^{(2)})^\complement\right|<\sum_{k=m}^\infty\frac{1}{2^{k+1}}=\frac{1}{2^m},
$$
so that indeed,
$$
F_1\subset F_2\subset\ldots\subset\mathcal{M},\quad\lim\limits_{m\to\infty}\left|F_m^\complement\right|=0.
$$
Set
$$
U\doteq\sum_{k=1}^\infty\left(H_k^{(1)}+H_k^{(2)}\right)=\sum_{n=1}^\infty\alpha_n\varphi_n.
$$
From (\ref{IndPart1}) and (\ref{IndPart2}) it follows that
$$
\sum_{k=1}^\infty\left(\|H_k^{(1)}\|_1+\|H_k^{(2)}\|_1\right)<1,
$$
therefore $U\in L^1(\mathcal{M})$. Moreover,
$$
\lim_{k\to\infty}\left\|\sum_{n=1}^{M_{2k+1}-1}\alpha_n\varphi_n-U\right\|_1=\lim_{k\to\infty}\left\|\sum_{\ell=k+1}^\infty\left(H_\ell^{(1)}+H_\ell^{(2)}\right)\right\|_1=0,
$$
so that by Lemma \ref{FourierCoeffUniqueLemma} the series $\sum_{n=1}^\infty\alpha_n\varphi_n$ is the Fourier series of the element $U$,
\begin{equation}
\alpha_n=c_n(U),\quad\forall n\in\mathbb{N}.\label{FourierCoeffId}
\end{equation}
Define the sequence $\{\delta_n\}_{n=1}^\infty\subset\mathbb{G}$ by
$$
\delta_n\doteq\begin{cases}
1 & \mbox{for}\quad n\in[M_{2k-1},M_{2k}),\quad k\in\mathbb{N}\\
\beta_n & \mbox{for}\quad n\in[M_{2k},M_{2k+1}),\quad k\in\mathbb{N}
\end{cases},\quad n\in\mathbb{N}.
$$

Now take any $f\in L^1(\mathcal{M})$ and find a subsequence $\{l_q\}_{q=1}^\infty\subset\mathbb{N}$ such that $f_{l_q}\xrightarrow[q\to\infty]{L^1}f$, and set $N_q\doteq M_{2l_q+1}-1$ for $q\in\mathbb{N}$. Then for every $m\in\mathbb{N}$ and every $q\in\mathbb{N}$ such that $l_q\ge m$ we have from (\ref{FourierPoly1}),(\ref{FourierPoly2}),(\ref{IndPart2}),(\ref{FmDef}) and (\ref{FourierCoeffId}) that
$$
\int\limits_{F_m}\left|\sum_{n=1}^{N_q}\delta_nc_n(U)\varphi_n(x)-f(x)\right|dx\le\int\limits_{F_m}\left|\sum_{n=1}^{N_q}\delta_n\alpha_n\varphi_n(x)-f_{l_q}(x)\right|dx+\|f_{l_q}-f\|_1
$$
$$
\le\int\limits_{E_{l_q}^{(2)}}\left|\sum_{j=1}^{l_q}\left(H_j^{(1)}(x)+Q_j^{(2)}(x)\right)-f_{l_q}(x)\right|dx+\|f_{l_q}-f\|_1\xrightarrow[q\to\infty]{}0,
$$
which completes the proof of the theorem. $\Box$
\end{proof}

\begin{remark} Mathematical objects introduced but not used in the proof of Theorem \ref{AsymTheorem1} will become part of the proof of Theorem \ref{AsymTheorem2}.
\end{remark}

Theorem \ref{AsymTheorem1} may leave the impression that the asymptotically conditionally universal function $U$ is very special. However, below we will prove a stronger result by showing that we can produce asymptotically conditionally universal functions by a small modification of an arbitrary integrable function. 

\begin{theorem}\label{AsymTheorem2} If the minimal system $\{\varphi_n\}_{n=1}^\infty$ possesses the asymptotic approximation property in $L^1(\mathcal{M})$ then there exist an integrable function $U\in L^1(\mathcal{M})$, which is
asymptotically conditionally universal for $L^1(\mathcal{M})$ w.r.t. $\{\varphi_n\}_{n=1}^\infty$, and a sequence of subsets $\{E_m\}_{m=1}^\infty\subset2^\mathcal{M}$, with
$$
E_1\subset E_2\subset\ldots\subset\mathcal{M},\quad\lim_{m\to\infty}\left|E_m^\complement\right|=0,
$$
such that
$$
\left(\forall g\in L^1(\mathcal{M})\right)\left(\forall m\in\mathbb{N}\right)\left(\exists V_m\in L^1(\mathcal{M})\right)
$$
$$
\left(\forall x\in E_m\right)\quad V_m(x)=g(x)\quad\wedge\quad(\exists\{\varepsilon_n\}_{n=1}^\infty\subset\mathbb{G})(\forall n\in\mathbb{N})\quad c_n(V_m)=\varepsilon_nc_n(U).
$$
\end{theorem}
\begin{proof} We retain all constructions and notations from the proof of Theorem \ref{AsymTheorem1}. Set
\begin{equation}
E_m\doteq\bigcap_{k=m}^\infty\left(E_k^{(1)}\cap E_k^{(2)}\right),\quad\forall m\in\mathbb{N},
\end{equation}
so that
$$
\left|E_m^\complement\right|\le\sum_{k=m}^\infty\left(\left\|E_k^{(1)}\right\|+\left\|E_k^{(2)}\right\|\right)<\sum_{k=m}^\infty\left(\frac1{2^{k+1}}+\frac1{2^{k+1}}\right)=\frac1{2^{m-1}}\xrightarrow[m\to\infty]{}0.
$$
Take an arbitrary $g\in L^1(\mathcal{M})$ , fix $m\in\mathbb{N}$  and choose a subsequence $\{k_q\}_{q=1}^\infty\subset\mathbb{N}+m$, such that
\begin{equation}
\left\|\sum_{r=1}^qf_{k_r}-g\right\|_1<\frac1{2^{q+2}},\quad\forall q\in\mathbb{N}.\label{kqDef}
\end{equation}
It follows that
\begin{equation}
\left\|f_{k_q}\right\|_1\le\left\|\sum_{r=1}^qf_{k_r}-g\right\|_1+\left\|\sum_{r=1}^{q-1}f_{k_r}-g\right\|_1<\frac1{2^{q+2}}+\frac1{2^{q+1}}<\frac1{2^q},\quad\forall q\in\mathbb{N}+1.\label{fkqEst}
\end{equation}
Set $\nu_0\doteq1$. We will use mathematical induction to construct a subsequence $\{\nu_q\}_{q=1}^\infty\subset\mathbb{N}$ and a sequence of functions $\{g_q\}_{q=1}^\infty\subset L^1(\mathcal{M})$ such that
\begin{equation}
g_q|_{E_m}=f_{k_q}|_{E_m},\quad\|g_{q+1}\|_1\le\frac{C+1}{2^{q-2}},\quad\left\|\sum_{r=1}^q\left[g_r-\left(\sum_{k=\nu_{r-1}}^{\nu_r-1}H_k+Q_{\nu_r}^{(1)}+H_{\nu_r}^{(2)}\right)\right]\right\|_1<\frac3{2^q},\quad\forall q\in\mathbb{N},\label{InductiveHyp}
\end{equation}
where
$$
H_k\doteq H_k^{(1)}+H_k^{(2)},\quad\forall k\in\mathbb{N}.
$$
For every $q\in\mathbb{N}$ choose the number $\nu_q>\nu_{q-1} +m$ such that
\begin{equation}
\left\|f_{k_q}+\sum_{r=1}^{q-1}\left[g_r-\left(\sum_{k=\nu_{r-1}}^{\nu_r-1}H_k+Q_{\nu_r}^{(1)}+H_{\nu_r}^{(2)}\right)\right]-f_{\nu_q}\right\|_1<\frac1{2^{q+1}},\label{nuqDef}
\end{equation}
and set
\begin{equation}
g_q\doteq f_{k_q}+\hat g_{\nu_q}-f_{\nu_q}.\label{gqDef}
\end{equation}
That $g_q|_{E_m}=f_{k_q}|_{E_m}$ is clear from $\hat g_{\nu_q}|_{E_{\nu_q}^{(1)}}=f_{\nu_q}|_{E_{\nu_q}^{(1)}}$ in (\ref{IndPart1}). Next we check that
$$
\left\|\sum_{r=1}^q\left[g_r-\left(\sum_{k=\nu_{r-1}}^{\nu_r-1}H_k+Q_{\nu_r}^{(1)}+H_{\nu_r}^{(2)}\right)\right]\right\|_1
$$
$$
=\left\|g_q-\sum_{k=\nu_{q-1}}^{\nu_q-1}H_k-Q_{\nu_q}^{(1)}-H_{\nu_q}^{(2)}+\sum_{r=1}^{q-1}\left[g_r-\left(\sum_{k=\nu_{r-1}}^{\nu_r-1}H_k+Q_{\nu_r}^{(1)}+H_{\nu_r}^{(2)}\right)\right]\right\|_1
$$
$$
=\left\|f_{k_q}+\hat g_{\nu_q}-f_{\nu_q}-\sum_{k=\nu_{q-1}}^{\nu_q-1}H_k-Q_{\nu_q}^{(1)}-H_{\nu_q}^{(2)}+\sum_{r=1}^{q-1}\left[g_r-\left(\sum_{k=\nu_{r-1}}^{\nu_r-1}H_k+Q_{\nu_r}^{(1)}+H_{\nu_r}^{(2)}\right)\right]\right\|_1
$$
$$
\le\left\|f_{k_q}+\sum_{r=1}^{q-1}\left[g_r-\left(\sum_{k=\nu_{r-1}}^{\nu_r-1}H_k+Q_{\nu_r}^{(1)}+H_{\nu_r}^{(2)}\right)\right]-f_{\nu_q}\right\|_1+\left\|\hat g_{\nu_q}-Q_{\nu_q}^{(1)}\right\|_1+\sum_{k=\nu_{q-1}}^{\nu_q}\left\|H_k\right\|_1
$$
\begin{equation}
<\frac1{2^{q+1}}+\frac1{2^{\nu_q+1}}+\frac1{2^{\nu_{q-1}-1}}\le\frac3{2^q},\quad\forall q\in\mathbb{N},\label{IntEst1}
\end{equation}
where we used (\ref{nuqDef}), (\ref{IndPart1}) and (\ref{IndPart2}) combined with $\nu_q\ge q$. From (\ref{IntEst1}), (\ref{nuqDef}) and (\ref{fkqEst}) we establish that
$$
\left\|f_{\nu_q}\right\|_1\le\left\|f_{k_q}+\sum_{r=1}^{q-1}\left[g_r-\left(\sum_{k=\nu_{r-1}}^{\nu_r-1}H_k+Q_{\nu_r}^{(1)}+H_{\nu_r}^{(2)}\right)\right]-f_{\nu_q}\right\|_1+\left\|f_{k_q}\right\|_1
$$
$$
+\left\|\sum_{r=1}^{q-1}\left[g_r-\left(\sum_{k=\nu_{r-1}}^{\nu_r-1}H_k+Q_{\nu_r}^{(1)}+H_{\nu_r}^{(2)}\right)\right]\right\|_1<\frac1{2^{q+1}}+\frac1{2^q}+\frac3{2^{q-1}}=\frac{15}{2^{q+1}},\quad\forall q\in\mathbb{N}+1.
$$
In conjunction with (\ref{fkqEst}), (\ref{gqDef}) and (\ref{IndPart1}), this implies
$$
\left\|g_q\right\|_1\le\left\|f_{k_q}\right\|_1+\left\|\hat g_{\nu_q}\right\|_1+\left\|f_{\nu_q}\right\|_1<\frac{C+1}{2^{q-3}},\quad\forall q\in\mathbb{N}+1.
$$
Thus, the properties (\ref{InductiveHyp}) are established.

Now set
$$
V_m\doteq\sum_{q=1}^\infty g_q.
$$
From (\ref{InductiveHyp}) we see that
$$
\sum_{q=1}^\infty\left\|g_q\right\|_1<\left\|g_1\right\|_1+\sum_{q=2}^\infty\frac{C+1}{2^{q-3}}<\infty,
$$
which shows that $V_m\in L^1(\mathcal{M})$. Moreover, from (\ref{InductiveHyp}) and (\ref{kqDef}) it follows that
$$
V_m|_{E_m}=\sum_{q=1}^\infty g_q|_{E_m}=\sum_{q=1}^\infty f_{k_q}|_{E_m}=g|_{E_m}.
$$
Define the sequence of $\mathbb{G}$-values (signs) $\{\varepsilon_n\}_{n=1}^\infty\subset\mathbb{G}$ by
$$
\varepsilon_n\doteq\begin{cases}
\beta_n & \mbox{for}\quad n\in[M_{2\nu_q-1},M_{2\nu_q}),\quad q\in\mathbb{N}\\
1 & \mbox{else}
\end{cases},\quad\forall n\in\mathbb{N}.
$$
Then by (\ref{InductiveHyp}) and by the construction of the function $U$ in the proof of Theorem \ref{AsymTheorem1} we have
$$
\left\|\sum_{n=1}^{M_{2\nu_q+1}-1}\varepsilon_nc_n(U)\varphi_n-V_m\right\|_1\le\left\|\sum_{r=1}^q\left[g_r-\left(\sum_{k=\nu_{r-1}}^{\nu_r-1}H_k+Q_{\nu_r}^{(1)}+H_{\nu_r}^{(2)}\right)\right]\right\|_1+\sum_{r=q+1}^\infty\left\|g_r\right\|_1
$$
$$
<\frac3{2^q}+\frac{C+1}{2^{q-3}}\xrightarrow[q\to\infty]{}0,
$$
which by Lemma \ref{FourierCoeffUniqueLemma} implies that $c_n(V_m)=\varepsilon_nc_n(U)$ for $n\in\mathbb{N}$. The proof is complete. $\Box$
\end{proof}

\begin{remark} If $\mathcal{M}=(\mathcal{M},\Sigma,|\cdot|)$ is in fact a Borel measure space on a topological space, and the subset $E$ in Definition \ref{UniApproxPropDef} is closed then the sets $\{F_m\}_{m=0}^\infty$ and $\{E_m\}_{m=0}^\infty$ in Theorem \ref{AsymTheorem1} and Theorem \ref{AsymTheorem2} can be considered closed.
\end{remark}

\section{Notes on classical systems}

That Definition \ref{UniApproxPropDef} of the universal approximation property is not vacuous is clear from the following lemma. Here we choose the classical case $\mathbb{G}=\{\pm1\}$.

\begin{lemma}\label{TrigUniApproxPropLemma} Let $\Omega=L^p([-\pi,\pi])$, $p\in(0,1)$, $\mathcal{X}=L^1([-\pi,\pi])$ and consider the trigonometric system
$$
\varphi_n(x)=e^{\imath nx},\quad\forall x\in[-\pi,\pi],\quad\forall n\in\mathbb{N}.
$$
Then the system $\{\varphi_n\}_{n=1}^\infty$ possesses the universal approximation property in $\Omega$.
\end{lemma}
This lemma is implied by Lemma 4 in \cite{GaGr21}. Although the statement of Lemma 4 in \cite{GaGr21} has $\epsilon=\delta$ in the language of our Definition \ref{UniApproxPropDef}, it is clear from Lemma 3 in \cite{GaGr21} that the same conclusions hold if $\epsilon$ and $\delta$ are independent.

Now Theorem \ref{TheoremAbstract} and Remark \ref{L1MRemark} imply the following important result.

\begin{corollary}\label{AlmostUnivCorr} There exists a function $U\in L^1([-\pi,\pi])$ which is almost universal for $L^p([-\pi,\pi])$, $p\in[0,1)$, as well as $M([-\pi,\pi])$, w.r.t. the trigonometric system.
\end{corollary}

\begin{remark}\label{UnivNotExistRemark} The statement in Corollary \ref{AlmostUnivCorr} is in a certain sense exhaustive, since an integrable function universal for any $L^p([-\pi,\pi])$, $p\in(0,1)$, in the \textit{usual} sense with respect to the trigonometric system does \textit{not} exist.
\end{remark}
Indeed, if such a function $U$ existed then there would exist a subsequence of partial sums of the Fourier series of $U$ that converges to, say, $2U$ in $L^p([-\pi,\pi])$. But by Kolmogorov's theorem (see Theorem 2 in Section \S21 of Chapter VII in \cite{Bar61}), the Fourier series of $U$ converges to $U$ in $L^p([-\pi,\pi])$, leading to a contradiction.

To our knowledge, Corollary \ref{AlmostUnivCorr} appears to be the first ever almost universality result in the literature. Note that conditionally universal functions for $L^p$, $p\in(0,1)$, were constructed in \cite{GaGr21} w.r.t. the trigonometric system, and in \cite{Gri20} w.r.t the Walsh system.

As for the asymptotic approximation property, it was proven in Lemma 3 of \cite{GaGr18} that the trigonometric system satisfies Definition \ref{AsymApproxPropDef} in the form given by Lemma \ref{AsymApproxLemma} above for $\mathbb{G}=\{\pm1\}$. Thus, Theorem \ref{AsymTheorem1} implies the following.

\begin{corollary}\label{AsymCondUniCorr} There exists a function $U\in L^1([-\pi,\pi])$ which is asymptotically conditionally universal for $L^1([-\pi,\pi])$ w.r.t. the trigonometric system.
\end{corollary}

\begin{remark} The statement in Corollary \ref{AsymCondUniCorr} is again exhaustive in a certain sense, because, as mentioned before, there exists \textit{no} integrable function which is \textit{conditionally} universal for $L^1([-\pi,\pi])$ w.r.t. any minimal system.
\end{remark}

Finally, Theorem \ref{AsymTheorem2} implies the following

\begin{corollary}\label{AsymCondUniGenCorr}  For every  $\epsilon>0$  there exists a set  $E$ with $|E^\complement|<\epsilon$, such that for every function $f\in L^1([-\pi,\pi])$  one can find a function $U\in L^1([-\pi,\pi])$ with $U|_E=f|_E$, which is asymptotically conditionally universal for $L^1([-\pi,\pi])$ w.r.t. the trigonometric system.
\end{corollary}

Corollary \ref{AsymCondUniCorr} and Corollary \ref{AsymCondUniGenCorr} appear to be new for the trigonometric system. For the Walsh system asymptotically conditionally universal functions were constructed in \cite{Gri20}. In fact, one can construct an integrable function of two variables which is asymptotically conditionally universal with respect to the double Walsh system. For other systems the question of existence of asymptotically conditionally universal functions remains open. In particular, it is unknown to us whether an integrable function exists on the 2-sphere which is asymptotically conditionally universal with respect to the spherical harmonics. In this respect it is worth mentioning that the method of proof of Theorem 1 in our paper \cite{AGR21} allows to construct a series in the spherical harmonics, such that the sequence of its partial sums is dense in the space of integrable functions on a subset of the sphere with a compliment of arbitrarily small measure. If we could find a suitable sequence of signs such that the above constructed series, after the change the signs of its terms, becomes the Fourier series of an integrable function, that function would become asymptotically conditionally universal with respect to the spherical harmonics.

\section*{Acknowledgements}

This work was supported by the RA MSECS Committee of Science, in
the frames of the research project \ 21AG-1A066. Z.A. and M.R. were supported by the FWO Odysseus 1 grant G.0H94.18N: Analysis and Partial Differential Equations and by the Methusalem programme of the Ghent University Special Research Fund (BOF) (Grant number 01M01021). M.R. was also supported by the EPSRC grants EP/R003025/2 and EP/V005529/1.


\begin{thebibliography}{9}

\bibitem{AGR21}
Z. Avetisyan, M. Grigoryan, M. Ruzhansky.
\textit{Approximations in $L^1$ with convergent Fourier series}.
Mathematische Zeitschrift 299, 2021, Pages 1907-1927.

\bibitem{Bar61}
N. Bary.
\textit{Trigonometric Series} (in Russian).
Gos. Izdat. Fiz. Mat. Lit., 1961.

\bibitem{BGNP08}
F. Bayart, K.-G. Gro{\ss}e-Erdmann, V. Nestoridis and C. Papadimitropoulos.
\textit{Abstract theory of universal series and applications.}
Proc. London Math Soc. 96(2), 2008, Pages 417-463.

\bibitem{BeBu02}
R. Beattie and H.-P. Butzmann.
\textit{Convergence Structures and Applications to Functional Analysis}.
Springer Dordrecht, 2002.

\bibitem{Bir29}
G. Birkhoff.
\textit{D\'emonstration d'un th\'eor\'eme \'el\'ementaire sur les fonctions enti\'eres}.
C. R. Acad. Sci. Paris 189, 1929, Pages 473-475.

\bibitem{EpMu14}
S. Episkoposian and J. M\"uller.
\textit{Universality properties of Walsh–Fourier series}.
Monatshefte f\"ur Mathematik 175, 1973, Pages 511–518.

\bibitem{Fek14}
M. Fekete.
\textit{Untersuchungen \"uber absolut summable Reihen, mit Anwendungen auf Dirichletsche und Fouriersche Reihen} (in Hungarian).
Math. \'es term\'esz. 32, 1914, Pages 389–425.

\bibitem{GaGr21}
L. Galoyan and M. Grigoryan.
\textit{Functions that are universal with respect to the trigonometric system}.
Izv. Math. 85(2), 2021, Pages 3-21.

\bibitem{GaGr18}
L. Galoyan and M. Grigoryan.
\textit{On the universal functions}.
Journal of Approx. Theory. 225, 2018, Pages 191-208.

\bibitem{GeNa99}
G. Gevorgyan and K. Navasardyan.
\textit{On Walsh series with monotone coefficients}.
Izv. Math. 63(1), 1999, Pages 37-55.

\bibitem{Gri99}
M. Grigorian.
\textit{On the representation of functions by orthogonal series in weighted spaces}.
Studia Math. 134(3), 1999, Pages 207-216.

\bibitem{GrSa16}
M. Grigoryan and A. Sargsyan.
\textit{On the universal function for the class $L^{p}[0,1]$, $p\in(0,1)$}.
Journal of Func. Anal. 270(8), 2016, Pages 3111-3133.

\bibitem{Gri17}
M. Grigoryan.
\textit{On the universal and strong property related to Fourier-Walsh series}.
Banach Journal of Math. Anal. 11(3), 2017, Pages 698-712.

\bibitem{Gri20}
M. Grigoryan.
\textit{Functions with universal Fourier-Walsh series}.
Sb. Math. 211(6), 2020, Pages 850-874.

\bibitem{Gri20a}
M. Grigoryan.
\textit{Universal Fourier Series.}
Math. Notes 108(2), 2020, Pages 282-285.

\bibitem{Gri22}
M. Grigoryan.
\textit{On Universal Fourier Series in the Walsh System}.
Sib. Math. J 63(2), 2022, Pages 868-882.

\bibitem{GrSa17}
M. Grigoryan, A. Sargsyan.
\textit{Universal function for a weighted space $\mathit{L_{\mu}^{1}[0,1]}$}.
Positivity 21(2), 2017, Pages 1425-1451.

\bibitem{Gro87}
K.-G. Gro{\ss}e-Erdmann.
\textit{Holomorphe Monster und Universelle Funktionen}.
Mitt. Math. Semin. Giessen. 176, 1987, Pages 1-84.

\bibitem{Iva90}
V. Ivanov.
\textit{Representation of functions by series in metric symmetric spaces without linear functionals}
Proc. Steklov Inst. Math. 189, 1990, Pages 37-85.

\bibitem{Kro91}
V. Krotov.
\textit{Smoothness of the universal Marcinkiewicz functions and universal trigonometric series}.
Izv. Uni. Math. 8, 1991, Pages 26-31.

\bibitem{Kol25}
A. Kolmogoroff.
\textit{Sur les fonctions harmoniques conjug\'ees et les s\'eries de Fourier}.
Fund. Math. 7(1), 1925, Pages 24-29.

\bibitem{Luz12}
N. Luzin.
\textit{On the main theorem of integral calculus} (in Russian).
Sb. Math. 28(2), 1912, Pages 262-294.

\bibitem{Luh86}
W. Luh.
\textit{Universal approximation properties of overconvergent power series on open sets}.
Analysis 6, 1986, Pages 191-207.

\bibitem{Mac52}
G. MacLane.
\textit{Sequences of derivatives and normal families}.
J. Anal. Math. 2, 1952, Pages 72-28.

\bibitem{Mar35}
J. Marcinkiewicz.
\textit{Sur les nombres d\'eriv\'es}.
Fund. Math. 24, 1935, Pages 305-308.

\bibitem{MeNe01}
A. Melas and V. Nestoridis.
\textit{On various types of universal Taylor series}.
Compl. Var. 44, 2001, Pages 245-258.

\bibitem{Men64}
D. Menshov.
\textit{On universal sequences of functions} (in Russian).
Sb. Math. 65(2), 1964, Pages 272-312.

\bibitem{Men42}
D. Menchoff.
\textit{Sur la convergence uniforme des s\'eries de Fourier}.
Rec. Math. [Math. Sb.] N. S. 11(53), 1942, Pages 67-96.

\bibitem{Nes96}
V. Nestoridis.
\textit{Universal Taylor series}.
Ann. Inst. Fourier 46, 1996, Pages 1293-1306.

\bibitem{Tal60}
A. Talalian.
\textit{On the universal series with respect to rearrangements} (in Russian).
Izv. Akad. Nauk SSSR, Ser. Math. 24, 1960, Pages 567-604.

\end{thebibliography}
\end{document}